\documentclass[12pt]{amsart}

\usepackage{latexsym, amssymb}

\newcommand{\be}{\begin{equation}}
\newcommand{\ee}{\end{equation}}
\newcommand{\beq}{\begin{eqnarray}}
\newcommand{\eeq}{\end{eqnarray}}

\newtheorem{prop}{Proposition}[section]
\newtheorem{theo}[prop]{Theorem}
\newtheorem{lemm}[prop]{Lemma}
\newtheorem{coro}[prop]{Corollary}
\newtheorem{rema}[prop]{Remark}

\def\begeq{\begin{equation}}
\def\endeq{\end{equation}}

\begin{document}

\title{Gauss map of translating solitons of mean curvature flow}

\begin{abstract}
In this short note we study Bernstein's type theorem of translating solitons whose images of their Gauss maps are contained in compact subsets in an open hemisphere of the standard $\mathbf{S}^n$ (see Theorem \ref{mainresult}). As a special case we get a classical Bernstein's type theorem in minimal submanifolds in $\mathbf{R}^{n+1}$ (see Corollary \ref{Bernsteintheorem}).
\end{abstract}

\renewcommand{\subjclassname}{\textup{2000} Mathematics Subject Classification}

\author{Chao Bao$^\dag$,
  Yuguang Shi$^\dag$}

\address{Chao Bao, Key Laboratory of Pure and Applied mathematics, School of Mathematics Science, Peking University,
Beijing, 100871, P.R. China.} \email{chbao@126.com}

\address{Yuguang Shi, Key Laboratory of Pure and Applied mathematics, School of Mathematics Science, Peking University,
Beijing, 100871, P.R. China.} \email{ygshi@math.pku.edu.cn}

\thanks{$^\dag$ Research partially supported by   NSF grant of China  10990013.}

\date{2013}
\maketitle

\markboth{Chao Bao,   Yuguang Shi}{}

 \maketitle

 \section{Introduction}

 Let $F_0 : \mathbf{\Sigma} \rightarrow \mathbf{R}^{n+1}$ be a smooth immersion of an
 n-dimensional hypersurface in $\mathbf{R}^{n+1}$, $n \geq 2$. The mean curvature flow is a one-parameter family of smooth
 immersions $F : \mathbf{\Sigma} \times [0,T) \rightarrow \mathbf{R}^{n+1}$ satisfying:
 \begin{equation} \label{MCF}
 \left\{
 \begin{array}{llll}
 \frac{\partial F}{\partial t}(p,t) = -H(p,t)\overrightarrow{\nu}  , p \in M, t \geq 0
 \\
 F(\cdot,0) = F_0
 \end{array}\right.
 \end{equation}
 where $-H(p,t) \overrightarrow{\nu}$ is the mean curvature vector, $\overrightarrow{\nu}(p,t)$ is the outer
 normal vector and $H(p,t)$ is the mean curvature respect to the normal vector $\overrightarrow{\nu}(p,t)$. It is easy
 to see that  the mean curvature of a convex surface is positive in our definition.

 If the initial hypersurface is compact, it is not hard to see that mean curvature flow must develop singularities in
 finite time. By the  blow up rate of second fundamental form, we can divided the singularities into
 two types, i.e. we say it type-1 singularity if there is a constant $C$ such that $\max\limits_{M_t} |A|^2 \leq
 \frac{C}{T-t}$ as $t \rightarrow T$, otherwise we say it of type-2, here $A$ is the second fundamental forms of  hypersurface at time $t$ . It is well-known that if a mean curvature
 flow develops type-1 singularities, we can get self-shrinking solutions after rescaling the surface near a singularity. Similarly, if the initial surface is mean convex and the singularity is of type-2, Huisken and Sinistrari \cite{HS} have proved that  any limiting flow is a  convex hypersurface which is a convex translating soliton. And it is not too difficult to see that a translating solitons is a hypersurface in $\mathbf{R}^{n+1}$ satisfying certain nonlinear equations (see (\ref{Sol})below, and we only consider codimension $1$ case in this paper.), it plays  an important role in analysis of singularities in mean curvature flow (see \cite{Ham, HS}, for examples). On the other hand, translating solitons can also be regarded as  natural generalizations of minimal hypersurface in $ \mathbf{R}^{n+1}$. With these facts in mind, it is natural to ask if Bernstin's type result is still true for translating solitons.

 Inspired by work of \cite{Ch, LW, X, XDY, W} etc, we focus on  investigation of uniqueness translating solitons solutions through their Gauss maps. We are able to get a Bernstein type property of translating solitons. Namely,  if  the Gauss map image of a translating soliton lies in a compact subset of a open hemisphere of $\mathbf{S}^n$, then it must be a hyperplane.  For self shrinkers similar results have been obtained by Xin, Ding,Yang \cite{XDY}(Theorem 3.2).

 Before we state our main results, we need to express some basic facts and notations first.

 A \emph{translating soliton} is a solution to (\ref{MCF}) translating in the direction of a constant vector $T$ in $\mathbf{R}^{n+1}$, more precisely, we call $F: \mathbf{\Sigma}^n \rightarrow\mathbf{R}^{n+1}$ is a translating soliton if $<T , \overrightarrow{\nu}(p)> = -H(p)$, here $\overrightarrow{\nu}(p)$ and $H(p)$ is defined as above. For simplicity, we will identify $F(\mathbf{\Sigma})$ and $\mathbf{\Sigma}$ in the sequel, and simply say $\mathbf{\Sigma}$ is a translating soliton.

 Let $\mathbf{\Sigma }$ be a translating soliton . We always take the outer normal vector through out this
 paper, denote the induced metric by $g = \{g_{ij}\}$, the surface measure by $d\mu$, the second fundamental
 form by $A = \{h_{ij}\}$, and mean curvature by $H = g^{ij}h_{ij}$. We then denote by $\lambda_1 \leq \cdots
 \leq \lambda_n$ the principal curvature, i.e. the eigenvalue of the matrix $(h^i_j) = (g^{ik}h_{kj})$. It is
 obviously that $H = \lambda_1 + \cdots + \lambda_n$. In addition, $|A|^2 = \lambda_1^2 + \cdots + \lambda_n^2$
 will denote the squared norm of the second fundamental form.

Let $u : \mathbf{\Sigma} \rightarrow \mathbf{S}^n$ be the Gauss map of the translating soliton $\mathbf{\Sigma}$, and the image $u(x)$ be the outer normal vector of $\mathbf{\Sigma}$.

\begin{theo}\label{mainresult}
Let $\mathbf{\Sigma }\subset\mathbf{R}^{n+1}$ be a n-dimensional complete translating soliton with bounded mean curvature. If the image of Gauss map $u$ of $\mathbf{\Sigma}$ lies in a ball $B_\Lambda^{S^n}(y_0)$ of $\mathbf{S}^n$, where $\Lambda < \frac{\pi}{2}$, then $\mathbf{\Sigma}$ must be a hyperplane.
\end{theo}

As it was mentioned before that a complete minimal hypersurface is a stationary solution of mean curvature flow, surely it is also a translating soliton, then we can also get a Bernstein type result for minimal hypersurfaces as following one. We wonder it is a well-known result, however, we cannot find the exact reference for it.

\begin{coro}\label{Bernsteintheorem}
  Let $\mathbf{\Sigma }\subset\mathbf{R}^{n+1}$ be a complete minimal hypersurface and image of Gauss map $u$ of $\mathbf{\Sigma}$ lies in a ball $B_\Lambda^{S^n}(y_0)$ of $\mathbf{S}^n$, where $\Lambda < \frac{\pi}{2}$, then $\mathbf{\Sigma}$ must be a hyperplane.
\end{coro}

\begin{rema}
By assuming $\mathbf{\Sigma}$ has Euclidean volume growth and the image under the Gauss map omits a neighbourhood of $\bar{\mathbf{S}}_{+}^{n-1}$, Jost,Xin,Yang could prove a similar result (see  Theorem 6.6 in \cite{JXY}). We suspect that our assumption on the image of Gauss map implies Euclidean volume growth of $\mathbf{\Sigma}$.
\end{rema}

In the remain part of the paper we will first derive some useful formulae for translating solitons and then give a proof of Theorem \ref{mainresult}.

\section{Gauss map of translating soliton}
Let $V = {V^\alpha}$ be the tangential part of $T$. Then the normal component must be $-HN^\alpha$
to solve the mean curvature flow. In local coordinates

 \begin{equation*}
   V^{\alpha} = V^i \nabla_{i}F^{\alpha}
 \end{equation*}

 where $\{V^i\}$ is a tangent vector on M, and the unit outer normal $\overrightarrow{N} = \{N^{\alpha}\}$.

 Take covariant derivative on the above equality for $i = 1, \cdots , n$.

 \begin{align*}
   0& = \nabla_i T^{\alpha} = \nabla_i(V^j \frac{\partial F^{\alpha}}{\partial x_j} - H N^{\alpha})
   \\
   & = (\nabla_i V^j) \frac{\partial F^{\alpha}}{\partial x_j} + V^j ( - h_{ij}N^{\alpha}) - (\nabla_i H) N^{\alpha} - H h_{ij}g^{jk} \frac{\partial F^{\alpha}}{\partial x_k}
 \end{align*}

 Then equating tangential and normal components we find that
\begin{equation}
 \left\{
 \begin{array}{llll}
   \nabla_i V^j = H h_{ik}g^{kj}
   \\
   \nabla_i H + h_{ij}V^j = 0
  \end{array}\right.                                \label{Sol}
 \end{equation}

We use the previous definition of Gauss map $u: \Sigma \rightarrow S^n$. The pull back under $u$ of the tangent bundle $TS^n$ to a bundle over $u$ is denoted by $u^{-1} T S^n$. $T \Sigma $ and $N\Sigma$ denote the tangent bundle and the normal bundle of $\Sigma$ respectively. It is easy to see that $u^{-1} TS^n$ is isometric to the tensor product $T\Sigma \times N\Sigma$.

By definition, the mean curvature $H = Tr u_{\ast}$ is the trace of $u_{\ast}$ respect to the Riemannian metric on $T\Sigma$. We denote $du = u_{\ast}$, and obviously $du$ is a cross section on $u^{-1}TS^n$.

So we can get that $\nabla H = \nabla Tr du = Tr \nabla du$. Now, let us compute exactly what $Tr \nabla du$ is. Denote $\Gamma$ and $\tilde{\Gamma}$ the Christoffel symbols on $\Sigma$ and $S^n$, $\{x^i\}$ and $\{y^{\alpha}\}$ the local coordinates on $\Sigma$ and $S^n$ respectively.

\begin{align*}
  \nabla du & = \nabla_j(du) dx^j = \nabla_j(\frac{\partial u^{\alpha}}{\partial x^i} dx^i \otimes \frac{\partial}{\partial y^{\alpha}}) dx^j\\
  & =  (\frac{\partial^2 u^{\alpha}}{\partial x^i \partial x^j}dx^i \otimes \frac{\partial}{\partial y^{\alpha}} - \Gamma^i_{lj} \frac{\partial u^{\alpha}}{\partial x^i}dx^l \otimes \frac{\partial}{\partial y^{\alpha}} + \tilde{\Gamma}^{\sigma}_{\alpha \beta} \frac{\partial u^{\alpha}}{\partial x^i}\frac{\partial u^{\beta}}{\partial x^j}dx^i \otimes \frac{\partial}{\partial y^{\sigma}}) dx^j
\end{align*}

So,
\begin{equation}
  Tr \nabla du = \Delta u^{\alpha} \frac{\partial}{\partial y^{\alpha}} + g^{ij} \tilde{\Gamma}^{\alpha}_{\beta \sigma} \frac{\partial u^{\beta}}{\partial x^i} \frac{\partial u^{\sigma}}{\partial x^j} \frac{\partial}{\partial y^{\alpha}}
\end{equation}

We denote
\begin{equation*}
\tau^{\alpha}(u) = \Delta u^{\alpha}  + g^{ij} \tilde{\Gamma}^{\alpha}_{\beta \sigma} \frac{\partial u^{\beta}}{\partial x^i} \frac{\partial u^{\sigma}}{\partial x^j}.
\end{equation*}

By the previous calculations, we get $\tau(u) = \nabla H$

Choosing an orthonormal frame $\{e_i\}_{i=1}^n$ on $\Sigma$. From (\ref{Sol}) and Weingarten formula:
\begin{equation*}
  \nabla_i u = h_{ij}e_j.
\end{equation*}
we get $\nabla H = - du(V)$. Then we have the following lemma:

\begin{lemm}
  The Gauss map $u$ of a translating soliton $\Sigma$ forms a quasi-harmonic map,i.e.
  \begin{equation}
    \tau(u) = - du(V)                                        \label{Har}
  \end{equation}
  where $\tau(u)$ is defined as above, $V$ is a tangent vector field on $\Sigma$.
\end{lemm}

\begin{lemm}
 For the quasi-harmonic equation of Gauss map $u$, we have the following Bochner formula:
 \begin{equation}
   \Delta|\nabla u|^2 = 2|\nabla du|^2 - 2|\nabla u|^4 - <V, \nabla|\nabla u|^2>.
 \end{equation}
\end{lemm}

\begin{proof}
  Denote $R_{ijkl}$ and $K_{\alpha \beta \gamma \sigma}$ be the curvature operator on $\Sigma$ and $S^n$ both with induced metric respectively. In general, we have the Bochner formula (See Lemma 3.1 in \cite{L}):
  \begin{equation}
    \Delta|\nabla u|^2 = 2|\nabla du|^2 + 2<d\tau(u) , du> + 2R_{ij}<u_i,u_j> - 2K_{\alpha \beta \gamma \sigma}u^{\alpha}_i u^{\beta}_j u^{\gamma}_i u^{\sigma}_j,
  \end{equation}

  By Gauss equation for hypersurface, we have
  \begin{equation*}
    R_{ij} = Hh_{ij} - h_{ik}h_{kj}
  \end{equation*}
  \begin{equation*}
    K_{\alpha \beta \gamma \sigma} = \delta_{\alpha \gamma} \delta_{\beta \sigma} - \delta_{\alpha \sigma} \delta_{\beta \gamma}
  \end{equation*}

  Because $u$ is Gauss map and is outer normal vector, Weingarten formula gives $u_i = h_{ij}e_j $. Together with $(\ref{Sol})$, We compute

\begin{align*}
  <d \tau(u) ,du> & = <du,d(du(V))> = <du,du(\nabla V)> + <du , \nabla_{V}du>  \\
   & = HA(\nabla u, \nabla u) + \frac{1}{2}<V, \nabla|\nabla u|^2>   \\
\end{align*}
Using Gauss equation, we have
\begin{align*}
  R_{ij}<u_i, u_j>
  & = HA(\nabla u, \nabla u) - h_{ik}h_{kj}h_{ip}h_{jp}      \\
  & = HA(\nabla u, \nabla u) - \sum\limits_{i,j}(\sum\limits_{k} h_{ik}h_{kj})^2.
\end{align*}
and
\begin{align*}
  K_{\alpha \beta \gamma \sigma}u^{\alpha}_i u^{\beta}_j u^{\gamma}_i u^{\sigma}_j
  & = \delta_{\alpha \gamma} \delta_{\beta \sigma} u^{\alpha}_i u^{\beta}_j u^{\gamma}_i u^{\sigma}_j - \delta_{\alpha \sigma} \delta_{\beta \gamma}u^{\alpha}_i u^{\beta}_j u^{\gamma}_i u^{\sigma}_j      \\
  & = |\nabla u|^4 - \sum\limits_{i,j}(<u_i , u_j>)^2 \\
  & = |\nabla u|^4 - \sum\limits_{i,j}(\sum\limits_{k} h_{ik}h_{kj})^2.
\end{align*}

The assertion follows from the above equalities.
\end{proof}

\begin{lemm}
  On any ball $B_{\Lambda}^{S^n}(y_0)$ of $S^n$, $\Lambda < \frac{\pi}{2}$, let $\rho$ be the distance function from $y_0$ on $S^n$, we define  $\varphi(y) = 1 - \cos \rho(y)$ on $B_{\Lambda}^{S^n}(y_0)$, then $\varphi$ satisfies the following properties:

  (1). There exists a constant $b$, such that $0 \leq \varphi < b < 1$;

  (2). $\frac{d\varphi}{d\rho} = \sin \rho$;

  (3). $Hess \varphi = (\cos \rho) I$, where $Hess \varphi$ is the hessian of $\varphi$, and $I$ is the identity matrix.
\end{lemm}
\begin{proof}
  (1) and (2) hold obviously, so we only need to prove (3). Applying Proposition 2.20 of \cite{GW} (or see \cite{Ch}),
  \begin{equation*}
    D^2 \rho = \frac{\cos \rho}{\sin \rho}(dS^2 - d\rho \otimes d\rho)
  \end{equation*}
  where $dS^2$ is the metric tensor on $S^n$. It is easy to check that $D^2 \varphi = \varphi' D^2 \rho + \varphi'' d\rho \otimes d\rho$, thus we get $D^2 \varphi = (\cos \rho) dS^2$, i.e. $\varphi_{ij} = (\cos \rho) \delta_{ij}$.
\end{proof}

\emph{Proof of Theorem \ref{mainresult}}.   Choosing a convex function on $B_{\Lambda}^{S^n}(y_0)$ as above, by direct computation we have
\begin{equation}
  \Delta \varphi (u(x)) = Hess \varphi (\nabla u, \nabla u) +<D\varphi,  \tau(u)> = \cos\rho |\nabla u|^2 + <V, \nabla \varphi>
\end{equation}

 Define $\phi(x) = \frac{|\nabla u|^2(x)}{(b - \varphi(u(x)))^2}$. Then

\begin{equation*}
  \nabla \phi (x) = \frac{\nabla |\nabla u|^2}{(b-\varphi)^2} + \frac{2|\nabla u|^2 \nabla \varphi}{(b - \varphi)^3}
\end{equation*}

\begin{align}
  \Delta \phi (x) & = \frac{\Delta |\nabla u|^2}{(b - \varphi)^2} + \frac{4<\nabla \varphi , \nabla |\nabla u|^2>}{(b - \varphi)^3} + \frac{2\Delta \varphi |\nabla u|^2}{(b - \varphi)^3} + \frac{6 |\nabla \varphi|^2 |\nabla u|^2}{(b - \varphi)^4}                                                     \notag   \\
  & = \frac{2|\nabla du|^2 +  <V, \nabla \varphi> - 2 |\nabla u|^4}{(b - \varphi)^2} + \frac{4<\nabla \varphi , \nabla |\nabla u|^2>}{(b - \varphi)^3}                                    \notag  \\
  & + \frac{2 \cos\rho |\nabla u|^4 + 2<V , \nabla \varphi |\nabla u|^2> }{ (b - \varphi)^3} + \frac{6 |\nabla\varphi|^2 |\nabla u|^2}{(b - \varphi)^4}       \label{Phi1}
\end{align}

Because
\begin{equation*}
 < V, \nabla \phi>  =  \frac{<V, \nabla |\nabla u|^2>}{(b-\varphi)^2} + \frac{2|\nabla u|^2 < V , \nabla \varphi>}{(b - \varphi)^3}
\end{equation*}

\begin{equation*}
 \frac{<\nabla \varphi, \nabla \phi>}{b - \varphi}  =  \frac{<\nabla \varphi , \nabla |\nabla u|^2>}{(b-\varphi)^3} + \frac{2|\nabla u|^2 |\nabla \varphi|^2 }{(b - \varphi)^4}
\end{equation*}

\begin{equation*}
  \frac{2|\nabla du|^2}{(b - \varphi)^2} + \frac{2|\nabla \varphi|^2 |\nabla u|^2}{(b - \varphi)^4} \geq \frac{4|\nabla \varphi||\nabla u||\nabla du|}{(b - \varphi)^3}
\end{equation*}

Then (${\ref{Phi1}}$) becomes
\begin{align}
  \Delta \phi (x) & \geq  \frac{2\cos\rho|\nabla u|^4}{(b - \varphi)^3} - \frac{2|\nabla u|^4}{(b - \varphi)^2} + \frac{2<\nabla \varphi , \nabla \phi>}{b - \varphi} +  <V, \nabla \phi>         \notag \\
  & \geq 2\cos\rho(b-\varphi)\phi^2 - 2(b - \varphi)^2 \phi^2 + \frac{2<\nabla \varphi , \nabla \phi>}{b - \varphi} +  <V, \nabla \phi>                                \label{Phi2}
\end{align}

Denote $r^2= |F(x)|^2$ be the distance function in $\mathbf{R}^{n+1}$ from the point $x \in \Sigma$ to the origin. $|\nabla r| \leq |\nabla F| \leq \sqrt{n-1}$. It is easy to check that $\Delta r^2 = 2(n-1) + 2<\overrightarrow{H} , F>$, here and in the sequel, $\nabla$ and $\Delta$ is gradient and Laplacian operator on $\mathbf{\Sigma}$ with induced metric. Note that the mean curvature of $\Sigma$ is bounded, we have

\begin{equation*}
  \Delta r^2 \leq C(1 + r).
\end{equation*}
where $C$ is only depend on $n$ and the upper bound of $|H|$.

Assume the origin $0 \in F(\Sigma)$,we denote $F = (R^2 - r^2)^2 \phi$, then there exists a point $x_0 \in \mathbf{\Sigma }\cap \mathbf{B}^{n+1}_R(0)$, such that $\nabla F(x_0) = 0$, $\Delta F (x_0) \leq 0$. Here $\mathbf{B}^{n+1}_R(0)$ is the ball with $0$ as the center and $R(>r)$ as radius in $\mathbf{R}^{n+1}$.

By a direct computation, we get
\begin{equation*}
  \nabla F = (R^2 - r^2)^2 \nabla \phi - 4r(R^2 - r^2)\phi \nabla r = 0
\end{equation*}

\begin{align*}
  \Delta F & = (R^2 - r^2)^2 \Delta \phi - 8 r(R^2 - r^2)\nabla \phi \nabla r + 8 r^2\phi|\nabla r|^2 - 2(R^2 -r^2)\phi \Delta r^2                        \\
  & \leq 0
\end{align*}
So,
\begin{equation}
  \frac{\nabla \phi}{\phi} = \frac{4r\nabla r}{R^2 -r^2}                          \label{Phi3}
\end{equation}

\begin{equation}
  \frac{\Delta \phi}{\phi} - \frac{2(C+1)r}{R^2 -r^2} - \frac{Cr^2}{(R^2 -r^2)^2} \leq 0     \label{Phi4}
\end{equation}
where $C$ is only depend on $n$ and $H$.

Then from ({\ref{Phi2}}), ({\ref{Phi3}}), and ({\ref{Phi4}}), we get
\begin{align*}
  & 2\cos\rho(b-\varphi)\phi  - 2(b - \varphi)^2 \phi + \frac{2<\nabla \varphi , \frac{\nabla \phi}{\phi}>}{b - \varphi} +  <V, \frac{\nabla \phi}{\phi}>   \\
  & - \frac{4n}{R^2 -r^2} - \frac{24r^2}{(R^2 -r^2)^2}   \\
  & = [2\cos\rho(b-\varphi) - 2(b - \varphi)^2] \phi  + \frac{<4r\nabla r , V>}{R^2 -r^2} +  \frac{8r<\nabla \varphi , \nabla r>}{(b - \varphi)(R^2 - r^2)}  \\
  & -  \frac{2(C+1)r}{R^2 -r^2} - \frac{Cr^2}{(R^2 -r^2)^2}  \\
  & \leq 0
\end{align*}

Because $T$ is a constant vector, and $V$ is its tangent part on $\Sigma$, then the norm of $V$ is bounded. By this the previous inequality becomes
\begin{align*}
  & [2\cos\rho(b-\varphi) - 2(b - \varphi)^2] \phi  - \frac{CR}{R^2 -r^2} -  \frac{8r|\nabla u|}{(b - \varphi)(R^2 - r^2)}  \\
  & -  \frac{2(C+1)r}{R^2 -r^2} - \frac{Cr^2}{(R^2 -r^2)^2}  \\
  & \leq 0
\end{align*}

By the definition of $F$, we have
\begin{align*}
   2(b-\varphi)(\cos\rho - (b - \varphi)) F -8R F^{\frac{1}{2}} - CR^3 - CR^2  \leq 0
\end{align*}

By
\begin{equation*}
  \cos\rho - (b - \varphi) = 1-b >0 , b - \varphi \geq b - \varphi(\Lambda)
\end{equation*}

we get

\begin{align*}
   F -CR F^{\frac{1}{2}} - CR^3 - CR^2  \leq 0
\end{align*}
where $C$ only depend on $n$, the upper bound of $|H|$, the norm of $V$, $b$, $\varphi$, and is independent of $R$.

Then,
\begin{equation}
  \sup\limits_{B_{\frac{R}{2}}^{n+1}(0) \cap \Sigma}F^{\frac{1}{2}}(x) \leq F^{\frac{1}{2}}(x_0) \leq C(R^{\frac{3}{2}} + R)
\end{equation}

So,
\begin{equation}
  \sup\limits_{B_{\frac{R}{2}}^{n+1}(0) \cap \Sigma} \frac{|\nabla u|(x)}{b-\varphi(u(x))} \leq C(\frac{1}{R^{\frac{1}{2}}} + \frac{1}{R}).
\end{equation}

By taking $R$ into infinity, we know that the Gauss map must be constant, so Theorem 1.1 follows.   $\Box$

\end{document}